\documentclass{iasr}
\renewcommand\thisnumber{0}
\renewcommand\thisyear {2011}
\renewcommand\thismonth{October}
\renewcommand\thisvolume{00}
\usepackage[dvips]{graphicx}
\usepackage{amsthm}

\begin{document}

\title{On Baer Invariants of Triples of Groups}
\author[Zohreh Vasagh, Hanieh Mirebrahimi and Behrooz Mashayekhy]{Zohreh Vasagh,
      Hanieh Mirebrahimi, and Behrooz Mashayekhy\corrauth}
 \address{Department of Pure Mathematics,
 Center of Excellence in Analysis on Algebraic Structures,
  Ferdowsi University of Mashhad,
  P.O.Box 91775-1159,
  Mashhad, Iran.}
 \corraddr{Behrooz Mashayekhy, Department of Pure Mathematics,
 Center of Excellence in Analysis on Algebraic Structures,
  Ferdowsi University of Mashhad,
  P.O.Box 91775-1159,
  Mashhad, Iran.
  Email: \tt bmashf@um.ac.ir}
\received{}

\newtheorem{thm}{Theorem}[section]
 \newtheorem{cor}[thm]{Corollary}
 \newtheorem{lem}[thm]{Lemma}
 \newtheorem{prop}[thm]{Proposition}
 \newtheorem{defn}[thm]{Definition}
 \newtheorem{rem}[thm]{Remark}
\newtheorem{Example}{Example}[section]

\begin{abstract}
In this paper, we develop the theory of Baer invariants for triples of groups. First, we focus on the general properties of the Baer invariant of triples. Second, we prove
that the Baer invariant of a triple preserves direct limits of directed systems of triples of groups. Moreover, we present a structure for the nilpotent multiplier of a triple
of the free product in some cases. Finally, we give some conditions in which the Baer invariant of a triple is a torsion group.
\end{abstract}

\keywords{Baer invariant; Triple of groups; Simplicial groups.}
\ams{57M07, 20J05, 55U10}
\maketitle

\section{Introduction and Motivation}
\label{sec1}
Throughout the article, let $\mathcal{V}$ be an arbitrary variety of groups defined by a set of laws $V$ and consider the functor ${V(-)}$ between groups which takes a group to its verbal subgroup. By a triple of groups $(G,M,N)$ we mean a group $G$ with two normal subgroups $N$, $M$. A homomorphism of triples $(G_1,M_1,N_1) \rightarrow (G_2,M_2,N_2)$ is a group homomorphism $G_1 \rightarrow G_2$ that sends $M_1$ into $M_2$ and $N_1$ into $N_2$.

Finding a topological interpretation for algebraic concepts  can be useful to solve some algebraic problems.
The Schur multiplier, $M(G)$, of an arbitrary group $G$ can be described as the second integral homology $H_2(X)$ of a connected
CW-space $X$, where $X$ has the fundamental group $G$ and contractible universal cover \cite{hop}. The Baer invariant of $G$, $\mathcal{V}M(G)$, has a similar topological interpretation. For this, consider the first homotopy group, $\pi_1(K_./V(K_.))$, of the factor of a free simplicial resolution of the group $G$ \cite{fr}.

G. Ellis \cite{ellis} defined the Schur multiplier of a pair of groups $(G,N)$ denoted by $M(G,N)$, as a functorial abelian group whose principal feature is a natural exact sequence
$$\begin{array}{ll}\cdots &\rightarrow
M(G,N)\rightarrow M(G)\rightarrow M(G/N)\\&\rightarrow N/[N,G]\rightarrow G^{ab}\rightarrow
(G/N)^{ab}\rightarrow 0.\end{array}$$
The natural epimorphism $G\rightarrow G/N$ implies the following exact sequence of free simplicial groups
$$1\rightarrow\ker (\alpha) \rightarrow K_. \stackrel{\alpha}{\rightarrow}L_.\rightarrow 1,$$
where $K_.$ and $L_.$ are free simplicial resolutions of $G$ and $G/N$, respectively \cite{fr}.
The authors in \cite{mmv} extended the above notation to the Baer invariant of a pair of groups, $\mathcal{V}M(G,N)$, as the first homotopy group of the kernel of a map $\alpha_v$, $\pi_1(\ker\alpha_v)$, where $\alpha_v:K_./V(K_.)\rightarrow L_./V(L_.)$ is induced from the simplicial map $\alpha$. Also the authors obtained the following long exact sequence
$$\begin{array}{ll}\cdots &\rightarrow
\mathcal{V}M(G,N)\rightarrow \mathcal{V}M(G)\rightarrow \mathcal{V}M(G/N)\\&\rightarrow N/[NV^*G]\rightarrow G/V(G)\rightarrow
(G/N)/V(G/N)\rightarrow 0.\end{array}$$

 The Schur multiplier of a triple was first studied by Ellis \cite{ellis}, who introduced it as a functorial abelian group $M(G,M,N)$ which fits into the following natural exact sequence
$$
\begin{array}{ll}
\cdots&\rightarrow M(G,M,N)\rightarrow M(G,N)\rightarrow M({G}/{M},{MN}/{M})\\
&\rightarrow {M\cap N}/{[M\cap N,G][M,N]}\rightarrow{N}/{[N,G]}\rightarrow {MN}/{M[N,G]}\rightarrow 0.
\end{array}
$$
 In this article, we intend to extend the concept of the Schur multiplier of triples to the Baer invariant of triples of groups with respect to an arbitrary variety $\cal V$ using
 homotopy theory for simplicial groups and free simplicial resolutions of groups.
\begin{definition}
Consider the following natural commutative diagram of a triple of groups $(G,M,N)$
$$
\begin{array}{ccc}
G &\rightarrow &\frac{G}{N}\\ \downarrow &&\downarrow\\ \frac{G}{M}&
\rightarrow &\frac{G}{MN}
\end{array}
$$
which implies the following commutative diagram of simplicial groups
\begin{equation}\label{eqa}
\begin{array}{rcccccl}
&1&&1&&1&\\&\downarrow&&\downarrow&&\downarrow&\\1\rightarrow&\ker(\alpha_V,\gamma_V)&\rightarrow&\ker\beta_V&
\rightarrow&\ker\delta_V&\rightarrow1\\&\downarrow&&\downarrow&&\downarrow&\\
1\rightarrow& \ker\alpha_V&\rightarrow&\frac{K_.}{V(K_.)} &\stackrel{\alpha_V}{\rightarrow}& \frac{L_.}{V(L_.)}&\rightarrow1\\ &\downarrow&&{\beta_V}{\downarrow}\ \ \ &&{\delta_V}{\downarrow} \ \ \ &\\ 1\rightarrow& \ker\gamma_V&\rightarrow&\frac{R_.}{V(R_.)}&\stackrel{\gamma_V}{\rightarrow}&
 \frac{S_.}{V(S_.)}& \rightarrow 1 \\ &\downarrow&&\downarrow&&\downarrow&\\&1&&1&&1,&\end{array}
\end{equation}
where $K_.$, $L_.$, $R_.$ and $S_.$ are free simplicial resolutions of $G$,  $G/N$, $G/M$ and $G/MN$, respectively \cite{bar}.
We define the Baer invariant of the triple  $(G,M,N)$ with respect to a variety of groups $\mathcal{V}$ as follows
 $$\mathcal{V}M(G,M,N)=\pi_1(\ker(\alpha_V,\gamma_V)),$$
  where $\ker(\alpha_V,\gamma_V)$ is defined in diagram \eqref{eqa}.
  \end{definition}

The article is organized as follows. In Section 2, we give a brief review for simplicial groups and some of their properties needed for the rest. In Section 3, First we give some useful general results for computation of Baer invariant of triple of groups. In particular, we obtain long exact sequences that contain Baer invariants of triples and pairs of groups.
Second, we prove that the Baer invariant of a triple preserves direct limits of directed systems of triples of groups. Moreover, we study the behavior of $c$-nilpotent multiplier with respect to the free product in some cases. Finally, we obtain some conditions in which the $c$-nilpotent multiplier of a triple of groups is a torsion group.

\section{Preliminaries and Notation}
\label{sec2}
In this section we recall several concepts and results about the simplicial groups which will be needed in the sequel. A detailed can be found in \cite{cur} or \cite{geo}.

A \emph{simplicial group} $K_.$ is a sequence of groups $K_0, K_1, K_2,\ldots$ together with homomorphisms $d_i:K_n\rightarrow K_{n-1}$ (faces) and $s_i:K_n\rightarrow K_{n+1}$ (degeneracies), for each $0\leq i\leq n$, such that the following conditions hold:
\[ \begin{array}{lcl}
  d_j d_i  &= & \ \ \ d_{i-1} d_j \ \ \quad\text{for $ j<i$} \\
   s_j s_i & =& \ \ \ s_{i+1} s_j \ \ \quad\text{for $ j\leq i$}\\
d_j s_i&=&\begin{cases}
s_{i-1} d_j  &\text{for $ j<i$;} \\
identity &\text{for $ j=i,i+1$;} \\
s_i d_{j-1}&\text{for $ j>i+1$.} \\
\end{cases} \end{array}\]
A \emph{simplicial homomorphism} $f:K_.\rightarrow L_.$ is a sequence of homomorphisms $f_n:K_n\rightarrow L_n$, for each $n\geq 0$, such that $f\circ d_i=d_i\circ f$, and $f\circ s_i=s_i\circ f$, that is the following diagram commutes:
\[
\begin{array}{ccccc}
K_{n+1}  & \stackrel{s_i}{\leftarrow}  &{K_n} & \stackrel{d_i}{\rightarrow}& K_{n-1}\\
\tiny{f_{n+1}}\downarrow\ \ \ \ &     &\tiny{f_{n}}\downarrow\ \ \ \ &         &\tiny{f_{n-1}}\downarrow\ \ \ \ \\
L_{n+1}   & \stackrel{s_i}{\leftarrow}&  {L_n} & \stackrel{d_i}{\rightarrow}&  L_{n-1}.
\end{array}
\]

From a simplicial group $K_.$ we can derive a chain complex $(NK_.,\partial)$ which is called the \emph{Moore complex} as follows: $(NK)_n=\cap_{i=0}^{n-1}Ker d_i$ and $\partial_n:NK_n\rightarrow NK_{n-1}$ to be the restriction of $d_n$.

\begin{theorem}\label{sg} (\cite{cur}).
\begin{enumerate}
\item For every simplicial group $K_.$ the homotopy group $\pi_n(K_.)$ is abelian even for $n=1$.\label{a}
\item Every epimorphism  between simplicial groups is a fibration.\label{f}
\item  Let $K_.$ be a simplicial group, then $\pi_*(K_.)\cong H_*(NK_.)$.\label{m}
\item $H_n\big(N(K_.\otimes L_.)\big)\cong H_n\big(N(K_.)\otimes N(L_.)\big).$ \label{n}
\end{enumerate}
\end{theorem}

A simplicial group $K_.$ is said to be \emph{free} if each $K_n$ is a free group and degeneracy homomorphisms $s_i$'s send the free basis of $K_n$ into the free basis of $K_{n+1}$. A \emph{free simplicial resolution} of $G$ consists of a free simplicial group $K_.$ with $\pi_0(K_.)=G$ and $\pi_n(K_.)=o$ for $n\geq1$ (see \cite{in}).
\section{Baer Invariants of Triples}
\label{sec3}
 In this section, we study the behavior of Baer invariants of triples of groups. First, we give some basic and essential properties of the Baer invariant of a triple of groups.
\begin{theorem}
The Baer invariant of a triple of groups is a functor from the category of triples of groups to the category of abelian groups.
\end{theorem}

\begin{proof}
 Let $f:(G_1,M_1,N_1)\rightarrow(G_2,M_2,N_2)$ be a homomorphism of triples of groups, functorial property of free simplicial resolutions yields the proof completed.
 \end{proof}

\begin{theorem}
Let $(G,M,N)$ be a triple of groups and $\mathcal{V}$ be an arbitrary variety of groups defined by a set of laws $V$. Then
 $\mathcal{V}M(G,M,N)$ satisfies in the following long exact sequences
\begin{equation}\label{triple}\begin{array}{ll}\cdots &\rightarrow\mathcal{V}M(G,M,N)\rightarrow\mathcal{V}M(G,N)\rightarrow\mathcal{V}M(G/M,MN/M)\\
&\rightarrow\pi_0(\ker(\alpha_V,\gamma_V))\rightarrow{N}/{[NV^*G]}\rightarrow {MN}/{M[NV^*G]}\rightarrow 0\end{array}\end{equation}
and
$$\begin{array}{ll}\cdots &\rightarrow\mathcal{V}M(G,M,N)\rightarrow\mathcal{V}M(G,M)\rightarrow\mathcal{V}M(G/N,MN/N)\\
&\rightarrow\pi_0(\ker(\alpha_V,\gamma_V))\rightarrow{M}/{[MV^*G]}\rightarrow {MN}/{N[MV^*G]}\rightarrow 0.\end{array}$$
\end{theorem}
\begin{proof}
Consider Diagram \eqref{eqa} corresponding to triple of groups $(G,M,N)$. In \cite{fr} one can find the following isomorphisms
$$\begin{array}{rlcrl}
\pi_0(\ker\alpha_v)&\cong \frac{N}{[NV*G]}&\ \ \ \ \ \ \ &\pi_0(\ker\beta_v)&\cong \frac{M}{[MV*G]}\\
\pi_0(\ker\gamma_v)&\cong \frac{MN}{N[MV*G]}&\ \ \ \ \ \ \  &\pi_0(\ker\delta_v)&\cong \frac{MN}{M[NV*G]}\\
\pi_0(\frac{K_.}{V(K_.)})&\cong \frac{G}{V(G)}&\ \ \ \ \ \ \  &\pi_0(\frac{L_.}{V(L_.)})&\cong \frac{G/N}{V(G/N)}\\
\pi_0(\frac{R_.}{V(R_.)})&\cong \frac{G/M}{V(G/M)}&\ \ \ \ \ \ \  &\pi_0(\frac{S_.}{V(S_.)})&\cong \frac{G/MN}{V(G/MN)}.\\ \end{array}$$

 The left column of Diagram \eqref{eqa} yields the long exact sequence of homotopy groups, using the various isomorphisms we can rewrite it in the group theory language as the exact sequence \eqref{triple}. Also the top row of Diagram \eqref{eqa} induces the other long exact sequence.
\end{proof}
\begin{theorem}
 Let $G=MN$ and $\mathcal{V}$ be the variety of abelian group of exponent q. Then we have the following long exact sequence
  $$\begin{array}{ll}\cdots &\rightarrow\mathcal{V}M(G,M,N)\rightarrow \mathcal{V}M(G,N)\rightarrow \mathcal{V}M(G/M)\\ &\rightarrow {M\cap N}/{M\#_q N}\rightarrow{N}/{[NV^*G]}\rightarrow {G/M}/{V(G/M)}\rightarrow 1,\end{array}$$
 where $M\#_q N$ is the subgroup of $G$ generated by
$$[m,n]t^q;\ \ m\in M,\ n\in N\ and \ t\in M\cap N.$$
\end{theorem}
\begin{proof}
 Barja and Rodriguez in \cite{bar} obtained
$ \pi_0(\alpha_V,\gamma_V)\cong {M\cap N}/{M\#_q N}$
which gives the result.
\end{proof}

 The long exact sequence of \eqref{triple} implies the following theorem.
 \begin{theorem}\label{ps}
The Baer invariant of a pair of groups is a special case of the Baer invariant of a triple of groups i.e. $\mathcal{V}M(G,M,M)\cong \mathcal{V}M(G,M)\cong \mathcal{V}M(G,G,M)$ for any group $G$ and any normal subgroup $M$.
\end{theorem}
\begin{theorem}
Let $\{(G_i,M_i,N_i)\}_{i\in I}$ be a directed system of triples
of groups. Then
$${\mathcal{V}M(\varinjlim G_i,\varinjlim M_i,\varinjlim N_i)\cong \varinjlim \mathcal{V} M(G_i,M_i,N_i)}.$$
\end{theorem}
\begin{proof}
For any $i\in I$, let $K_{i_.}, L_{i_.}, R_{i_.}$ and $S_{i_.}$ be corresponding free simplicial resolutions of $G_i, G_i/N_i, G_i/M_i$ and $G_i/M_iN_i $, respectively. Assume the following diagram of simplicial groups
$$
\begin{array}{ccc}
\frac{K_{i_.}}{V(K_{i_.})} &\stackrel{\alpha_{V_i}}{\rightarrow}& \frac{L_{i_.}}{V(L_{i_.})}\\ {\beta_{V_i}}\downarrow\ \ \ \ & & \ \ \ \  \downarrow{\delta_{V_i}}\\ \frac{R_{i_.}}{V(R_{i_.})}&\stackrel
{\gamma_{V_i}}{\rightarrow}& \frac{S_{i_.}}{V(S_{i_.})}.
\end{array}
$$
 The authors in \cite{mmv} proved that $\varinjlim \pi_n(K_{i_.})\cong \pi_n(\varinjlim K_{i_.})$, where $K_{i_.}$ are simplicial groups. Hence $\varinjlim K_{i_.},\ \varinjlim L_{i_.},\  \varinjlim R_{i_.}$ and $\varinjlim S_{i_.}$ are simplicial groups corresponding to $\varinjlim G_i,\ \varinjlim G_i/N_i,\ \varinjlim G_i/M_i$ and $\varinjlim G_i/M_iN_i$, respectively.
Also, the authors \cite[Theorem 2.7]{mmv2} proved that $\varinjlim (\ker {\alpha_{V_i}})\cong  \ker \varinjlim({\alpha_{V_i}})$ and similarly $\varinjlim (\ker {\gamma_{V_i}})\cong  \ker \varinjlim({\gamma_{V_i}}).$
 Now consider the following commutative diagram:
$$\begin{array}{rcccccl}1\rightarrow&\ker ((\varinjlim{\alpha_{V_i}},\varinjlim{\gamma_{V_i}}))&\rightarrow&\ker \varinjlim({\gamma_{V_i}})&\rightarrow& \ker \varinjlim({\alpha_{V_i}})&\rightarrow1\\
&\downarrow&&\downarrow&&\downarrow&\\ 1\rightarrow&\varinjlim (\ker ({\alpha_{V_i}},{\gamma_{V_i}}))&\rightarrow&\varinjlim (\ker {\gamma_{V_i}})&\rightarrow&\varinjlim (\ker {\alpha_{V_i}})&\rightarrow1.\end{array}$$
Similar to the proof of \cite[Theorem 2.7]{mmv2}, the Five Lemma implies that
$$\varinjlim (\ker ({\alpha_{V_i}},{\gamma_{V_i}}))\cong  \ker ((\varinjlim{\alpha_{V_i}},\varinjlim{\gamma_{V_i}})).$$
Since homotopy groups of simplicial groups commute with direct limits \cite{mmv}, the following isomorphisms hold
  $$\begin{array}{cc}\mathcal{V}M(\varinjlim G_i,\varinjlim M_i,\varinjlim N_i)&\cong\pi_1\big( \ker (\varinjlim{\alpha_{V_i}},\varinjlim{\gamma_{V_i}})\big) \\ \cong\pi_1\Big(\varinjlim \big(\ker ({\alpha_{V_i}},{\gamma_{V_i}})\big)\Big)&\cong  \varinjlim \mathcal{V} M(G_i,M_i,N_i).\end{array}$$
\end{proof}
Now we study the behavior of Baer invariants of triples of free products.
\begin{theorem}\label{2}
Let $(G_i,M_i,N_i)$ be triples
of groups, for $i=1,2$. Then the following isomorphism holds
$$ M(G_1 \ast G_2,\langle M_1\ast M_2\rangle^{G_1\ast G_2},\langle N_1\ast N_2\rangle^{G_1\ast G_2})\cong M(G_1,M_1,N_1)\oplus M(G_2,M_2,N_2).$$
\end{theorem}
 \begin{proof}
Van-Kampen theorem for simplicial groups implies that $K_{1_.}\ast K_{2_.}$, $L_{1_.}\ast L_{2_.}$, $R_{1_.}\ast R_{2_.}$ and $S_{1_.}\ast S_{2_.}$ are free simplicial resolutions corresponding to $G_1\ast G_2$, $(G_1/N_1\ast G_2/N_2)$, $(G_1/M_1\ast G_2/M_2) $ and $(G_1/M_1N_1\ast G_2/M_2N_2)$, respectively \cite{bur}.
Since $K_{1_.}$ and $K_{2_.}$ are free groups, we have $(K_{1_.}\ast K_{2_.})^{ab}\cong K_{1_.}^{ab}\oplus K_{2_.}^{ab}$. Hence we have the following commutative diagram
 $$\begin{array}{ccc}K_{1_.}^{ab}\oplus K_{2_.}^{ab}&\stackrel{\alpha_{1}\oplus\alpha_{2}}{\rightarrow}&L_{1_.}^{ab}\oplus L_{2_.}^{ab}\\
\downarrow& &\downarrow\\ R_{1_.}^{ab}\oplus R_{2_.}^{ab}&\stackrel{\gamma_{1}\oplus\gamma_{2}}{\rightarrow}&S_{1_.}^{ab}\oplus S_{2_.}^{ab}.
 \end{array} $$
  Consequently, $\ker({\alpha_{1}\oplus\alpha_{2}},{\gamma_{1}\oplus\gamma_{2}})\cong \ker(\alpha_{1},\gamma_{1})\oplus\ker(\alpha_{2},\gamma_{2})$ which completes the proof.

 \end{proof}
 \begin{lemma}\label{l}
 Let $G_1$ and $G_2$ be two groups with $(|G_1^{ab}|,|G_2^{ab}|)=1$. Then for all $c\geq 1$ the following isomorphism holds
 $$\pi_n(\frac{K_{1_.}*K_{2_.}}{\gamma_c(K_{1_.}*K_{2_.})})\cong\pi_n(\frac{K_{1_.}}{\gamma_c(K_{1_.})})
 \oplus\pi_n(\frac{K_{2_.}}{\gamma_c(K_{2_.})}),$$
 where $K_{1_.}$ and $K_{2_.}$ are free simplicial resolutions of $G_1$ and $G_2$, respectively.
 \end{lemma}
 \begin{proof}
First, we consider the following exact sequence which is proved in \cite{mmv}
\begin{equation}\label{ker} 1\rightarrow\ker{\phi_c}\rightarrow \frac{K_{1_.}\ast K_{2_.}}{\gamma_c(K_{1_.}\ast K_{2_.})}\rightarrow\frac{K_{1_.}\times K_{2_.}}{\gamma_c(K_{1_.}\times K_{2_.})}\rightarrow1.\end{equation}
Note that $\ker{\phi}_c$ satisfies in the following exact sequence
\begin{equation}\label{tensor} 1\rightarrow \frac{[L_{1_.},K_{2_.},_{c-2}F]^F}{[K_{1_.},K_{2_.},_{c-1}F]^F} \rightarrow \ker{\phi}_c\rightarrow\ker{\phi}_{c-1}\rightarrow 1,\end{equation}
where $F_.=K_{1_.}*K_{2_.}$. Moreover
 $$ \frac{[K_{1_.},K_{2_.},_{c-2}F]^F}{[K_{1_.},K_{2_.},_{c-1}F]^F}\cong\oplus\sum_{\substack{for\ some\ i+j=c}} {\underbrace{K_{1_.}^{ab}\otimes...\otimes K_{1_.}^{ab}}_{i-times}} \otimes\displaystyle{\underbrace{K_{2_.}^{ab}\otimes...\otimes K_{2_.}^{ab}}_{j-times}}.$$
Fix $n $, now by induction on $c$, we show that  $\pi_n(\ker\phi_c)$ is trivial.
For $c=2$, Theorem \ref{2} shows that $\ker\phi_c=0$.

Let $c>2$, using Theorem \ref{sg} we have
$$\begin{array}{ll}\pi_n(K_{1_.}^{ab}\otimes K_{2_.}^{ab})&\cong \oplus\sum_{\substack{i+j=n}}\pi_i(K_{1_.}^{ab})\otimes
\pi_{j}(K_{2_.}^{ab})\\&\oplus\sum_{\substack{i+j=n-1}}Tor(\pi_i(K_{1_.}^{ab}),
\pi_{j}(K_{2_.}^{ab}))\\&\cong \oplus\sum_{\substack{i+j=n}}H_{i+1}(G_1)\otimes
H_{j+1}(G_2)\\&\oplus\sum_{\substack{i+j=n-1}}Tor(H_{i+1}(G_1),H_{j+1}(G_2))=0 ,\end{array} $$
and
$$\begin{array}{ll}&\pi_n({\underbrace{K_{1_.}^{ab}\otimes...\otimes K_{1_.}^{ab}}_{i-times}} \otimes\displaystyle{\underbrace{K_{2_.}^{ab}\otimes...\otimes K_{2_.}^{ab}}_{j-times}})\cong\\& \oplus\sum_{\substack{r+s=n}}\pi_r(K_{1_.}^{ab}\otimes K_{2_.}^{ab})\otimes\pi_s({\underbrace{K_{1_.}^{ab}\otimes...\otimes K_{1_.}^{ab}}_{i-1-times}} \otimes\displaystyle{\underbrace{K_{2_.}^{ab}\otimes...\otimes K_{2_.}^{ab}}_{j-1-times}}) \\&
\oplus\sum_{\substack{r+s=n-1}}Tor(\pi_r(K_{1_.}^{ab}\otimes K_{2_.}^{ab}),\pi_s({\underbrace{K_{1_.}^{ab}\otimes...\otimes K_{1_.}^{ab}}_{i-1-times}} \otimes\displaystyle{\underbrace{K_{2_.}^{ab}\otimes...\otimes K_{2_.}^{ab}}_{j-1-times}}))\\
&=0.
\end{array}$$
Therefore the induction hypothesis and the exact sequence \eqref{tensor} imply that $\pi_n(\ker\phi_{c})$ is trivial for all $n$.

The exact sequence \eqref{ker} gives rise the following long exact sequence of homotopy groups
$$\pi_n(\ker{\phi_c})\rightarrow \pi_n(\frac{K_{1_.}\ast K_{2_.}}{\gamma_c(K_{1_.}\ast K_{2_.})})\rightarrow\pi_n(\frac{K_{1_.}\times K_{2_.}}{\gamma_c(K_{1_.}\times K_{2_.})})\rightarrow\pi_{n-1}(\ker{\phi_c}).$$
Hence the result holds.
\end{proof}
\begin{theorem}
Let $(G_i,M_i,N_i)$ be triples
of groups, for $i=1,2$, such that $(|G_1^{ab}|,|G_2^{ab}|)=1$. Then for all $c\geq 1$
$$ M^{(c)}(G_1 \ast G_2,\langle M_1\ast M_2\rangle^{G_1\ast G_2},\langle N_1\ast N_2\rangle^{G_1\ast G_2})\cong M^{(c)}(G_1,M_1,N_1)\oplus M^{(c)}(G_2,M_2,N_2).$$
\end{theorem}

\begin{proof}
Consider the same notation as the proof of Theorem \ref{2} and fix $c$. Applying the homomorphism
 $\frac{K_{1_.}\ast K_{2_.}}{\gamma_c(K_{1_.}\ast K_{2_.})}\rightarrow\frac{K_{1_.}\times K_{2_.}}{\gamma_c(K_{1_.}\times K_{2_.})}$ to the simplicial resolutions of pairs of groups $(G_i,N_i)$, for $i=1,2$, we get the following commutative diagram
$$\begin{array}{lclclclc}& \pi_{n+1}(\frac{L_{1_.}\ast L_{2_.}}{\gamma_c(L_{1_.}\ast L_{2_.})})&\rightarrow&\pi_n(\ker(\alpha_1*\alpha_2))&\rightarrow&\pi_n(\frac{K_{1_.}\ast K_{2_.}}{\gamma_c(K_{1_.}\ast K_{2_.})})&\rightarrow&\pi_n(\frac{L_{1_.}\ast L_{2_.}}{\gamma_c(L_{1_.}\ast L_{2_.})})\\ &\downarrow & &\downarrow & &\downarrow & &\downarrow \\& \pi_{n+1}(\frac{L_{1_.}\times L_{2_.}}{\gamma_c(L_{1_.}\times L_{2_.})})&\rightarrow&\pi_n(\ker(\alpha_1\times\alpha_2))&\rightarrow&\pi_n(\frac{K_{1_.}\times K_{2_.}}{\gamma_c(K_{1_.}\times K_{2_.})})&\rightarrow&\pi_n(\frac{L_{1_.}\times L_{2_.}}{\gamma_c(L_{1_.}\times L_{2_.})}).
\end{array}$$
Lemma \ref{l} and the Five Lemma imply that $\pi_n(\ker(\alpha_1*\alpha_2))\cong \pi_n(\ker(\alpha_1))\oplus\pi_n(\ker(\alpha_2))$. Similarly for the pair of groups $(G/M,MN/M)$ we have $\pi_n(\ker(\gamma_1*\gamma_2))\cong \pi_n(\ker(\gamma_1))\oplus\pi_n(\ker(\gamma_2))$. With a similar argument for the exact sequence \eqref{triple} we have the following commutative diagram
$$\begin{array}{clclclc}&\rightarrow&\pi_n(\ker(\alpha_1*\alpha_2),\ker(\gamma_1*\gamma_2))&\rightarrow
&\pi_n(\ker(\alpha_1*\alpha_2))&\rightarrow&\pi_n(\ker(\gamma_1*\gamma_2))\\& &\downarrow& &\downarrow& &\downarrow\\
&\rightarrow&\pi_n(\ker(\alpha_1,\gamma_1),\ker(\alpha_2,\gamma_2))&\rightarrow
&\pi_n(\ker(\alpha_1,\alpha_2))&\rightarrow&\pi_n(\ker(\gamma_1,\gamma_2)).\end{array}$$
Therefore we can prove that $\pi_n(\ker(\alpha_1*\alpha_2, \gamma_1*\gamma_2))\cong \pi_n(\ker(\alpha_1,\gamma_1))\oplus\pi_n(\ker(\alpha_2,\gamma_2)).$ Hence for $n=1$, we have
$$ M^{(c)}(G_1 \ast G_2,\langle M_1\ast M_2\rangle^{G_1\ast G_2},\langle N_1\ast N_2\rangle^{G_1\ast G_2})\cong M^{(c)}(G_1,M_1,N_1)\oplus M^{(c)}(G_2,M_2,N_2).$$
\end{proof}
\begin{theorem}\label{torsion}
Let $M^{(1)}(G,N)$, $M^{(1)}(G/N)$ and $H_3(G/N)$ be torsion groups. Then $M^{(c)}(G,N)$ is also a torsion group, for all $c\geq2.$
\end{theorem}
\begin{proof}
Let $K_.$ and $L_.$ be free simplicial resolutions of $G$ and $G/N$, respectively. Consider the epimorphisms
 $\frac{K_.}{\gamma_n(K_.)} \stackrel{\alpha_n}{\rightarrow} \frac{L_.}{\gamma_n(L_.)}$ and $ \frac{\gamma_n(K_.)}{\gamma_{n+1}(K_.)} \stackrel{\alpha_n / \alpha_{n+1}}{\rightarrow}\frac{\gamma_n(L_.)}{\gamma_{n+1}(L_.)}.$ By \cite[Claim 2]{el}  $\pi_1(\gamma_{n}(K_.)/\gamma_{n+1}(K_.))$ is a torsion group. In \cite{mmv} the authors obtained the following exact sequence
\begin{equation}\label{n}\pi_2(\frac{\gamma_{n}(L_.)}{\gamma_{n+1}(L_.)})\rightarrow\pi_1(\ker\alpha_n/\alpha_{n+1})
\rightarrow \pi_1(\frac{\gamma_{n}(K_.)}{\gamma_{n+1}(K_.)}).\end{equation}
 Note that $\pi_1(\ker\alpha_n/\alpha_{n+1})$ satisfies in the following exact sequence
\begin{equation}\label{nn}\pi_1(\ker\alpha_n/\alpha_{n+1})\rightarrow M^{(n+1)}(G,N)\rightarrow M^{(n)}(G,N).\end{equation}
With a similar argument of \cite[Claim 2]{el}  we can get the isomorphism $\pi_2(L^{ab}_.\otimes Q)\cong H_3(G/N)\otimes{Q}=0$ which leads to  $\pi_2(\gamma_{n}(L_.)/\gamma_{n+1}(L_.))\otimes{Q}\cong0$. Therefore $\pi_2(\gamma_{n}(L_.)/\gamma_{n+1}(L_.))$ is a torsion group. Thus \eqref{n} implies that $\pi_1(\ker\alpha_n/\alpha_{n+1})$ is a torsion group. Hence the result follows by induction from \eqref{nn}.
\end{proof}
Theorem \ref{torsion} can be extended to triples of the groups with the similar argument.
\begin{theorem}
Let $M^{(1)}(G,M,N)$, $M^{(1)}(G)$, $H_3(G/N)$, $H_3(G/M)$ and $H_4(G/MN)$ be torsion groups. Then $M^{(c)}(G,M,N)$ is a torsion group, for all $c\geq1.$
\end{theorem}
\begin{proof}
A detailed proof is quite lengthy, so we give a sketch of proof. Consider the following commutative diagrams which induce from the diagram of the free simplicial resolutions of triple of groups $(G,M,N)$ \eqref{eqa}
$$
\begin{array}{ccc}
\frac{K_.}{\gamma_n(K_.)} &\stackrel{\alpha_n}{\rightarrow}& \frac{L_.}{\gamma_n(L_.)}\\ \downarrow& &\downarrow \\ \frac{R_.}{\gamma_n(R_.)}&\stackrel
{\beta_n}{\rightarrow}& \frac{S_.}{\gamma_n(S_.)}
\end{array}\ \ \ \ \ \ \ \ \ \
\begin{array}{ccc}
\frac{\gamma_n(K_.)}{\gamma_{n+1}(K_.)}&\stackrel{\alpha_n / \alpha_{n+1}}{\rightarrow}&\frac{\gamma_n(L_.)}{\gamma_{n+1}(L_.)}\\ \downarrow& &\downarrow \\ \frac{\gamma_n(R_.)}{\gamma_{n+1}(R_.)}&\stackrel
{\beta_n/ \beta_{n+1}}{\rightarrow}& \frac{\gamma_n(S_.)}{\gamma_{n+1}(S_.)}.
\end{array}$$
With a similar argument of \cite{mmv2}, we can obtain the following exact sequence
$$\pi_1(\ker(\alpha_n/\alpha_{n+1},\beta_n/\beta_{n+1}))\rightarrow M^{(n+1)}(G,M,N)\rightarrow M^{(n)}(G,M,N)$$
which satisfies in exact sequence
$$\pi_2(\ker\beta_n/\beta_{n+1})\rightarrow\pi_1\ker((\alpha_n/\alpha_{n+1},\beta_n/\beta_{n+1}))\rightarrow \pi_1(\ker\alpha_n/\alpha_{n+1}).$$
Moreover, $\pi_2(\ker\beta_n/\beta_{n+1})$ fits in the following exact sequence
$$\pi_3(\frac{\gamma_{n}(S_.)}{\gamma_{n+1}(S_.)})\rightarrow\pi_2(\ker\beta_n/\beta_{n+1})
\rightarrow \pi_2(\frac{\gamma_{n}(R_.)}{\gamma_{n+1}(R_.)}).$$
By Theorem \ref{torsion}, we have $\pi_2(\ker\alpha_n/\alpha_{n+1})$ and $\pi_2(\frac{\gamma_{n}(R_.)}{\gamma_{n+1}(R_.)})$ are torsion groups. Hence the
result is proved if we show that $\pi_3(\frac{\gamma_{n}(S_.)}{\gamma_{n+1}(S_.)})$ is a torsion group.

With a similar argument of the proof of  Theorem \ref{torsion}, we can prove that if $H_4(G/MN)$ is a torsion group, then so is $\pi_3(\frac{\gamma_{n}(S_.)}{\gamma_{n+1}(S_.)})$.
\end{proof}


\end{document}